\documentclass[12pt,reqno]{amsart}

\usepackage[margin=1in]{geometry}
\usepackage[fleqn,tbtags]{mathtools}
\usepackage[shortlabels]{enumitem}
\usepackage{graphicx}
\usepackage{amssymb}
\usepackage{amsthm}
\usepackage{mathrsfs}
\usepackage{mlmodern}
\usepackage{eucal}
\usepackage{microtype}

\usepackage[nodisplayskipstretch]{setspace}
\usepackage{tikz}
\usetikzlibrary{arrows.meta,positioning,calc}
\usepackage{xcolor}

\usepackage[colorlinks, breaklinks=true, pagebackref=false]{hyperref}
\usepackage[nameinlink, capitalize, noabbrev]{cleveref}

\newtheorem{thm}{Theorem}[section]
\newtheorem{prop}{Proposition}[section]
\newtheorem{lem}{Lemma}[section]

\newtheorem{rem}{Remark}[section]

\newtheorem{prob}{Problem}[section]

\newcommand{\PSH}{\operatorname{PSH}}
\newcommand{\QPSH}{\operatorname{QPSH}}
\newcommand{\ddc}{dd^c}

\newcommand{\supp}{\operatorname{supp}}

\newcommand{\restr}[2]{\left.#1\right|_{#2}}

\title[Extension of K\"ahler currents]{On the extension of K\"ahler currents on compact complex manifolds}
\author[J. Ning]{{Jiafu Ning}}
\address{Jiafu Ning: \ School of Mathematics and Statistics\\ HNP-LAMA\\Central South University\\Changsha\\Hunan 410083\\ P. R. China}
\email{jfning@csu.edu.cn}
\author[K. Pang]{Kai Pang}
\address{Kai Pang: \ Institute of Mathematics\\Academy of Mathematics and Systems Science\\
	Chinese Academy of Sciences\\Beijing 100190\\ P. R. China}
\email{pangkai@amss.ac.cn}
\author[H. Sun]{Haoyuan Sun}
\address{Haoyuan Sun: School of Mathematical Sciences\\ Beijing Normal University\\ Beijing 100875\\ P. R. China}
\email{202531130037@mail.bnu.edu.cn}
\author[Z. Wang]{Zhiwei Wang}
\address{Zhiwei Wang: Laboratory of Mathematics and Complex Systems (Ministry of Education)\\ School of Mathematical Sciences\\ Beijing Normal University\\ Beijing 100875\\ P. R. China}
\email{zhiwei@bnu.edu.cn}

\author[X. Zhou]{Xiangyu Zhou}
\address{Xiangyu Zhou: Institute of Mathematics\\ Academy of Mathematics and Systems Science\\
	and Hua Loo-Keng Key Laboratory of Mathematics\\ Chinese Academy of Sciences\\ Beijing 100190\\ P. R. China}

\email{xyzhou@math.ac.cn}

\subjclass[2020]{32U05, 32Q15, 53C55}
\keywords{quasiplurisubharmonic function, K\"ahler current, extension theorem, compact K\"ahler manifold, complex submanifold}

\begin{document}
	
	\begin{abstract}
Let $(X,\omega)$ be a compact K\"ahler manifold and let $V\subset X$ be a closed complex submanifold. Coman-Guedj-Zeriahi proposed the problem: is every $\omega|_V$-plurisubharmonic function on $V$ the restriction of an $\omega$-plurisubharmonic function on $X$? In this paper, we solve this problem affirmatively, even for a compact Hermitian manifold.
		
	\end{abstract}
	
	\maketitle
	\tableofcontents
	\section{Introduction}
	
	Let $(X,\omega)$ be a compact K\"ahler manifold and $V\subset X$ a closed complex submanifold. Throughout this paper, we assume that $V\neq X$.  A function $u\colon X\to[-\infty,+\infty)$ is called quasi-plurisubharmonic (qpsh for short) if locally it can be written as the sum of a plurisubharmonic function and a smooth function. We fix the following notations: $d^c:=\frac{i}{2\pi}(\bar\partial-\partial)$ and $\omega_V:=\omega|_V$.
Define $\QPSH(X):=\{u: u\ \text{is qpsh on}\  X\}$, 
$\PSH(X,\omega) :=\{u\in\QPSH(X):\omega+\ddc u\geq0\}$ and $\PSH(X,\omega,\delta) :=\{u\in\QPSH(X):\omega+\ddc u\geq\delta\omega\}$.

	In this paper, we study  the following  problem raised by Coman-Guedj-Zeriahi.
	\begin{prob}[{\cite[p. 40]{CGZ13},\cite[Problem 37]{DGZ}}]\label{prob:1}
		Does the following hold
		\begin{align*}
			\mbox{PSH}(V,\omega|_V)=\mbox{PSH}(X,\omega)|_V?
		\end{align*}
	\end{prob}
	
	A related and strict problem asks whether every strictly $\omega_V$-plurisubharmonic potential on $V$ is the restriction of a strictly $\omega$-plurisubharmonic potential on $X$, or equivalently,  whether every K\"ahler current in $[\omega_V]$ extends to a K\"ahler current in $[\omega]$. Here  $[\omega_V]$ (resp. $[\omega]$) is the Bott-Chern cohomology class of $\omega_V$ (resp. $\omega$) in $H^{1,1}(V,\mathbb R)$ (resp.  $H^{1,1}(X,\mathbb R)$).
	
	Several cases were known. The Hodge case was obtained in \cite{CGZ13}; the real N\'eron--Severi case, including singular analytic subvarieties, was treated in \cite{CGZ22}. Smooth strictly plurisubharmonic functions extend by \cite[Proposition~2.1]{CGZ13}, see also \cite{Sch98}. Collins-Tosatti proved the strict extension theorem for currents with analytic singularities \cite{CT14} (see also \cite{DRWXZ26} for a refinement). Ning-Wang-Zhou proved the arbitrary-singularity strict theorem under the existence of a holomorphic retraction from a neighborhood of $V$ onto $V$ \cite{NWZ24}, see also the earlier tubular-neighborhood result by Wang-Zhou \cite{WZ}.
	
	In this paper, we give a positive answer to the \cref{prob:1} in  full generality.
	
	\begin{thm}\label{thm:main}
		Let $(X,\omega)$ be a compact Hermitian manifold and let $V\subset X$ be a closed complex submanifold. Then every $\varphi\in\PSH(V,\omega_V)$ admits an extension $\Phi\in\PSH(X,\omega)$ such that $\restr{\Phi}{V}=\varphi$.
		Equivalently, $\PSH(X,\omega)|_V=\PSH(V,\omega_V)$.
	\end{thm}
	
	As a direct application, we get the following theorem.
	
	\begin{thm}\label{thm:strict}
		Let $(X,\omega)$ and $V$ be as in \cref{thm:main}. Suppose that $\varphi\in\PSH(V,\omega_V,\varepsilon)$
		for some $\varepsilon>0$. Then there are $\varepsilon'>0$ and $\Phi\in\PSH(X,\omega,\varepsilon')$  such that
		$\restr{\Phi}{V}=\varphi$.
		Consequently, if $\omega$ is a K\"ahler metric, then every K\"ahler current in $[\omega_V]$ is the restriction of a K\"ahler current in $[\omega]$.
	\end{thm}
	\begin{rem}
		\cref{thm:main} and \cref{thm:strict} do not require $(X,\omega)$ to be compact K\"ahler.     To prove \cref{thm:main}, we follow the strategy of \cite{NWZ24} while introducing a new technique. The main novelty over previous results \cite{NWZ24} is that we do not insist on  simultaneously extending  the smooth approximations $\varphi_j$ of $\varphi$ on $V$ to a common neighborhood of $V$ with uniform estimate, and  instead, we relax the uniformity of extending data while preserving monotonicity. The  argument could apply to the non-smooth case, \textit{i.e.}, $V$ a subvariety of $X$. This will be studied in a subsequent paper.
	\end{rem}
	
	\subsection*{Acknowledgements}
	This research is supported by the National Key R\&D Program of China (Grant No. 2021YFA1002600 and  No. 2021YFA1003100).  J. Ning, Z. Wang and X. Zhou are partially supported by grants from the National Natural Science Foundation of China
(NSFC)(No. 12071485), (No. 12571085) and (No. 12288201) respectively. Z. Wang is also supported by the Fundamental Research Funds for the Central Universities.
	
	\section{Preliminaries}
Throughout this paper, we assume that
	$(X,\omega)$ is a compact Hermitian manifold with $\omega$ a Hermitian metric, $V\subset X$ be a complex submanifold, and $V\neq X$.
A function $u\colon X\to[-\infty,+\infty)$ is called quasi-plurisubharmonic (qpsh for short) if locally it can be written as the sum of a plurisubharmonic function and a smooth function. We fix the following notations: $d^c:=\frac{i}{2\pi}(\bar\partial-\partial)$ and $\omega_V:=\omega|_V$.
Define $\QPSH(X):=\{u: u\ \text{is qpsh on}\  X\}$, $\PSH(X,\omega) :=\{u\in\QPSH(X):\omega+\ddc u\geq0\}$ and $\PSH(X,\omega,\delta) :=\{u\in\QPSH(X):\omega+\ddc u\geq\delta\omega\}$.
	
	Following the convention in \cite{CGZ22}, if $V$ is disconnected we allow a function in $\PSH(V,\omega_V)$ to be identically $-\infty$ on some, but not all, connected components. The notation $\PSH(X,\omega)|_V$ refers to restrictions which are not identically $-\infty$ on all of $V$. In the strict class, no connected component is allowed to carry the value $-\infty$ identically.
	
	We repeatedly use the following standard facts.
	
	\begin{enumerate}[label=(\roman*)]
		\item If $u,v\in\PSH(U,\omega)$ on an open set $U$, then $\max\{u,v\}\in\PSH(U,\omega)$.
		Moreover, if $u,v\in \PSH(U,\omega,\delta)$, then $\max\{u,v\}\in \PSH(U,\omega,\delta)$.
		\item A decreasing sequence in $\PSH(X,\omega)$ either converges identically to $-\infty$ on each connected component, or its pointwise limit belongs to $\PSH(X,\omega)$.
	\end{enumerate}
	
	For a qpsh function $u$ whose restriction to $V$ is not identically $-\infty$ on any connected component, the restriction current is understood through local potentials:
	\[
	i^*(\omega+\ddc u):=\omega_V+\ddc(\restr{u}{V}),
	\]
	where $i\colon V\hookrightarrow X$ denotes the inclusion.

	The following approximation theorem is a consequence of  \cite[Theorem 2.3]{DPS01} or   \cite[Theorem 3.2]{DP04}, see also \cite[Theorem 1]{BK07}.
	
	\begin{thm}[c.f. \cite{DPS01,DP04,BK07}]\label{thm: appro}
		Let $(X,\omega)$ be a compact Hermitian manifold with $\omega$ a Hermitian metric. Let $\varphi\in \mbox{PSH}(X,\omega)$. Then there is a  sequence of smooth functions $\varphi_j\in \mbox{PSH}(X,\omega)$ deceasing to $\varphi$ on $X$.
	\end{thm}
	
	We need the  following lemma.
	\begin{lem}[c.f. \cite{DP04}]\label{reference function}
		There exists an exponentially continuous  qpsh function $F:X\rightarrow [-\infty, +\infty)$ which is smooth on $X\setminus V$, with logarithmic singularities along $V$, and such that $F\in \PSH(X,\omega, \varepsilon)$.
		By subtracting a large constant, we can make $F\leq 0$ on $X$.
	\end{lem}

	Recall that	a function $F:X\rightarrow [-\infty,+\infty)$ is said to be exponentially continuous, if $e^F:X\rightarrow [0,+\infty)$ is continuous.

Multiplying the above function $F$ by a cut-off function, we obtain the following lemma.
	
	\begin{lem}\label{lem:barrier}
		Let $W$ be an open neighbourhood of $V$.	There exist an exponentially continuous qpsh function $b_W\colon X\longrightarrow[-\infty,0]$
		and a constant $C_W\geq0$ such that $b_W$ is smooth on $X\setminus V$, with logarithmic singularities along $V$ and such that $b_W=0$ on a neighbourhood of $X\setminus W$ and $\ddc b_W\geq-C_W\omega$.
	\end{lem}
	
	\begin{proof}
		Choose open neighborhoods
		\[
		V\Subset W_0\Subset W_1\Subset W.
		\]
		
		Choose $\chi\in C_c^\infty(W_1)$ with $0\leq\chi\leq1$ and $\chi=1$ on $W_0$. Define
		\[
		b_W:=
		\begin{cases}
			\chi F,&\text{on }N,\\
			0,&\text{on }X\setminus\supp\chi,
		\end{cases}
		\]
		where $F$ is the qpsh chosen by \cref{reference function}.
		It is easy to check that $b_W$ satisfies the desired properties.
	\end{proof}

	\section{Local smooth extension}
	
	In this section, we introduce a local smooth extension, which will serve as a candidate in the global patching procedure.
	
	\begin{lem}\label{lem:local-smooth}
		Fix $\varepsilon>0$. For any  $0<\eta<\frac{\varepsilon}{2}$,  and  any  $f\in \mathcal C^\infty(V)\cap\PSH(V,\omega_V,\varepsilon)$, there are an open neighborhood $U_f$ of $V$ and a function $H_f\in \mathcal C^\infty(U_f)\cap\PSH(U_f,\omega,\eta)$ such that $\restr{H_f}{V}=f$ on $U_f$.
	\end{lem}

	To prove \cref{lem:local-smooth}, we need the following proposition due to Ning-Wang-Zhou.
	\begin{prop}[{\cite[Proposition 2.1]{NWZ24}}]\label{prop:hessian}Let $(X,\omega)$ be a complex $n$-dimensional  Hermitian manifold with a Hermitian metric $\omega$. Let $V\subset X$ be a complex submanifold of complex dimension $k<n$, and $h(z):=\mbox{dist}^2(z, V)$ be the square of the distance function  on $X$ with respect to the Reimannian metric induced by $\omega$.   Let $p\in V$ be an arbitrarily fixed point in $V$, then there is a holomorphic coordinate chart  $(U,z=(z_1,\cdots,z_k,z_{k+1},\cdots, z_n))$ centered at $p$ such that $U\cap V=\{z_{k+1}=\cdots=z_n=0\}$, $\omega=\sqrt{-1}\sum_{i,j=1}^ng_{i\bar j}dz_i\wedge d\bar z_j$ with $g_{i\bar j}(0)=\delta_{i\bar j}$ for $i,j=1,\cdots, n$, and
		\begin{align*}
			\frac{\partial^2h}{\partial z_i\partial \bar z_j}(0)=
			\begin{cases}
				0, \quad\quad i \; \text{or} \; j\leq k; \\
				\delta_{i\bar j}, \quad\;\  i,j>k.
			\end{cases}
		\end{align*}
	\end{prop}
	\begin{proof}[Proof of \cref{lem:local-smooth}]
		
		The proof is essentially contained in the proof of \cite[Proposition 2.1]{CGZ13}. We include a proof here for the sake of completeness.
		Let $h=\operatorname{dist}_\omega(\,\cdot\,,V)^2$. It is smooth on a tubular neighborhood of $V$ \cite{Ma}. Along $V$, use the  $\omega$-orthogonal $\mathcal C^\infty$ splitting
		\[
		T^{1,0}X|_V=T^{1,0}V\oplus N^{1,0}_{V/X}.
		\]
		Since $\omega $ is compatible with the complex structure $J$ of $X$, and $T^{1,0}V$ is $J$-invariant, it follows that $N^{1,0}_{V/X}$ is also $J$-invariant, which is isomorphic to the real normal bundle $N_{V/X}$.
		By \cref{prop:hessian}, for any $p\in V$, and $\xi\in T^{1,0}_pV$ and $\zeta\in N^{1,0}_{V/X,p}$, we have
		\begin{equation}\label{eq:distance-hessian}
			\ddc h(\xi+\zeta,\overline{\xi+\zeta})
			\geq \frac{1}{\pi}|\zeta|_\omega^2.
		\end{equation}
		Choose any smooth extension $\widetilde f$ of $f$ to a neighborhood of $V$ and put
		\[
		\Theta:=\omega+\ddc\widetilde f.
		\]
		The tangential block of $\Theta$ along $V$ is bounded below by $\varepsilon\omega_V$. Since $V$ is compact, for this fixed $f$ there is $M_f>0$ such that, along $V$,
		\begin{align*}
			|\Theta(\xi,\bar\zeta)|&\leq M_f|\xi|_\omega|\zeta|_\omega,\\
			\Theta(\zeta,\bar\zeta)&\geq-M_f|\zeta|_\omega^2.
		\end{align*}
		Hence, for $A>0$,
		\begin{align*}
			(\Theta+A\ddc h)(\xi+\zeta,\overline{\xi+\zeta})
			&\geq \varepsilon|\xi|_\omega^2
			-2M_f|\xi|_\omega|\zeta|_\omega
			+(Ac_0-M_f)|\zeta|_\omega^2\\
			&\geq \frac{\varepsilon}{2}|\xi|_\omega^2
			+\left(\frac{A}{\pi}-M_f-\frac{2M_f^2}{\varepsilon}\right)|\zeta|_\omega^2,
		\end{align*}
		where in the second inequality, we use the Young inequality $2ab\leq \frac{\varepsilon}{2}a^2+\frac{2}{\varepsilon}b^2$ with $a=|\xi|_\omega$ and $b=M_f|\zeta|_\omega$.
		Choose $A=A_f$ so that the last coefficient is at least $\varepsilon$. Then, along $V$,
		\[
		\omega+\ddc(\widetilde f+A_fh)\geq\frac{\varepsilon}{2}\omega.
		\]
		By compactness and continuity, after shrinking to a neighborhood $U_f$ this remains bounded below by $\eta\omega$. Set
		\[
		H_f:=\widetilde f+A_fh.
		\]
		Since $h|_V=0$, the restriction of $H_f$ is $f$.
	\end{proof}
	\begin{rem}
		It is worth mentioning that, in the proof of \cref{thm:main}, it suffices that $f$ admits a continuous extension to a strictly $\omega$-plurisubharmonic function on a neighborhood of $V$, which may also be established via a Richberg type argument similar as \cite{CT14}.
	\end{rem}
	
	
	\section{Dominated global patching lemma}
	
	In this section, we prove a dominated global patching lemma which will play the key role in the proof of \cref{thm:main}.
	
	\begin{lem}\label{lem:insertion}
		Let $G\in\PSH(X,\omega,\delta)$ for some $\delta>0$ and suppose that $G$ is smooth on a neighborhood of $V$. Let $H\in \mathcal C^\infty(U)\cap\PSH(U,\omega,\eta)$ on an open neighborhood $U$ of $V$ for some $\eta>0$. Suppose further that  $\restr{H}{V}<\restr{G}{V}$.
		Then, for every $0<\delta'<\min\{\delta,\eta\}$,
		there exists a global qpsh function $G'$ with the following properties:
		\begin{itemize}
			\item  $G'\leq G$ on $X$;
			\item $G'=H$ on a neighborhood of $V$;
			\item $G'\in \PSH(X,\omega,\delta')$.
		\end{itemize}
		In particular, $G'$ is smooth near $V$ and $\restr{G'}{V}=\restr{H}{V}$.
	\end{lem}
	
	\begin{proof}
		Since $G$ and $H$ are smooth near $V$ and $\restr{H}{V}<\restr{G}{V}$, there is an open set $O$ which
		\[
		V\Subset O\Subset U
		\]
		contained in the smoothness neighborhood of $G$, such that
		\begin{equation}\label{eq:HbelowG}
			H<G\quad\text{on }\overline O.
		\end{equation}
		Apply Lemma~\ref{lem:barrier} inside $O$. We obtain a function $b=b_O\leq0$, with pole set $V$, supported compactly in $O$, and satisfying
		\[
		\ddc b\geq-C\omega
		\]
		for some $C\geq0$.
		
		Choose $\lambda>0$ so small that
		\begin{equation}\label{eq:lambda-choice}
			\delta-\lambda C\geq\delta'.
		\end{equation}
		Set
		\[
		Q:=G+\lambda b.
		\]
		Then $Q\leq G$ and
		\[
		\omega+\ddc Q
		\geq(\delta-\lambda C)\omega
		\geq\delta'\omega.
		\]
		Define
		\begin{equation}\label{eq:inserted-function}
			G':=
			\begin{cases}
				\max\{H,Q\},&\text{on }O,\\
				Q,&\text{on }X\setminus O.
			\end{cases}
		\end{equation}
		Since  the function $b$ is exponentially continuous, so is $Q$. From the fact that $e^H|_V>0=e^Q|_V$, it follows from the continuity of $e^H $ and $e^Q$ that near $V$, $G'=H$. In particular $G'$ is smooth near $V$.  	Then it is easy to check that $G'$ satisfies the other desired properties. The proof of \cref{lem:insertion} is complete.
		
		
	\end{proof}
	
	
	\section{Strict monotone smooth approximation with a gap}
	In this section, we will prove a strict monotone smooth approximation with a gap, of a  $\omega_V$-psh function on $V$.
	
	\begin{lem}\label{lem:strict-gap-approx}
		Let $\varphi\in\PSH(V,\omega_V)$. After subtracting a constant from $\varphi$, there exists a sequence $\psi_j\in \mathcal C^\infty(V)$ such that
		\begin{itemize}
			\item $\psi_1\leq-1$, and $\psi_j\searrow\varphi$;
			\item $\psi_j\in \PSH(V,\omega_V,2^{-j})$;
			\item $\psi_{j-1}-\psi_j\geq 2^{-j}\quad(j\geq2)$.
		\end{itemize}
	\end{lem}
	
	\begin{proof}
		
		On a connected component of $V$, on which $\varphi\equiv-\infty$, set  $ \psi_j:=-j-3$.
		There the curvature lower bound is $\omega_V\geq2^{-j}\omega_V$, the consecutive gap is $1$, and $\psi_j\searrow-\infty$.
		On every component where $\varphi$ is not identically $-\infty$,
		\cref{thm: appro} gives smooth functions
		\[
		\varphi_{j}\in\PSH(V,\omega_V),
		\qquad
		\varphi_{j}\searrow {\varphi}.
		\]
		Subtracting  one common constant from $\varphi$ and all these approximants so that $\varphi_{1}\leq-3$.
		Put $a_j:=1-2^{-j}$ and define, $\psi_j:=a_j\varphi_{j}+2^{-j}$.
		Then
		\[
		\omega_V+\ddc\psi_j
		=2^{-j}\omega_V+a_j(\omega_V+\ddc\varphi_{j})
		\geq2^{-j}\omega_V,
		\]
		and $\psi_1\leq-1$. Moreover, for $j\geq2$,
		\begin{align*}
			\psi_{j-1}-\psi_j
			&=a_{j-1}(\varphi_{j-1}-\varphi_{j})
			+(a_{j-1}-a_j)\varphi_{j}
			+(2^{-(j-1)}-2^{-j})\\
			&=a_{j-1}(\varphi_{j-1}-\varphi_{j})
			+2^{-j}(1-\varphi_{j})\\
			&\geq2^{-j}.
		\end{align*}
		The convergence follows from $a_j\to1$, $2^{-j}\to0$, and $\varphi_{j}\searrow\varphi$, including at points where $\varphi=-\infty$.
		Combining the definitions over the disjoint components gives a smooth global sequence on $V$ satisfying the properties in \cref{lem:strict-gap-approx}.
	\end{proof}

	\section{Proof of Main Theorems}
	
	In this section, we give the proof of our main theorems.
	\begin{proof}[Proof of \cref{thm:main}]

		We may assume that $X$ is connected, while the general case follows componentwise. Subtract a constant from $\varphi$, take the sequence $\psi_j$ provided by Lemma~\ref{lem:strict-gap-approx}, and finally add the constant back.
		
		Set
		\[
		G_0:=0,
		\qquad
		\delta_0:=1.
		\]
		We construct inductively global qpsh functions $G_j$ and positive numbers $\delta_j$ for $j\geq1$ such that
		\begin{align}
			&G_j\leq G_{j-1},\label{eq:Gmonotone}\\
			&G_j\text{ is smooth on a neighborhood of }V,\label{eq:Gsmooth}\\
			&\restr{G_j}{V}=\psi_j,\label{eq:Gtrace}\\
			& G_j\in \PSH(X,\omega,\delta_j).\label{eq:Gpositive}
		\end{align}
		Suppose $G_{j-1}$ is constructed. By Lemma~\ref{lem:local-smooth} applied with $\varepsilon=2^{-j}$, there are a neighborhood $U_j$ of $V$, a smooth function $H_j$ on $U_j$, and a number $\eta_j>0$ such that
		\[
		\restr{H_j}{V}=\psi_j,
		\qquad
		\omega+\ddc H_j\geq\eta_j\omega\quad\text{on }U_j.
		\]
		Set
		\[
		\delta_j:=\frac12\min\{\delta_{j-1},\eta_j\}>0.
		\]
		For $j=1$,  \cref{lem:local-smooth} gives
		\[
		\restr{H_1}{V}=\psi_1\leq-1
		=\restr{G_0}{V}-1.
		\]
		For $j\geq2$, \cref{lem:local-smooth}  gives
		\[
		\restr{H_j}{V}=\psi_j
		\leq\psi_{j-1}-2^{-j}
		=\restr{G_{j-1}}{V}-2^{-j}.
		\]
		Apply Lemma~\ref{lem:insertion} with $G=G_{j-1}$, $H=H_j$, and target lower bound $\delta_j$.
       We denote the resulting function by $G_j$, then $\{G_j,\delta_j\}$ satisfies \eqref{eq:Gmonotone}--\eqref{eq:Gpositive}.
		
		The sequence $G_j$ is decreasing. Define
		\[
		\Phi:=\lim_{j\to\infty}G_j.
		\]
		Choose $p\in V$ with $\varphi(p)>-\infty$. Then
		\[
		G_j(p)=\psi_j(p)\longrightarrow\varphi(p)>-\infty,
		\]
		so the sequence does not converge identically to $-\infty$. Since every $G_j$ is $\omega$-psh, the decreasing-limit theorem gives $\Phi\in\PSH(X,\omega)$. Finally,
		\[
		\restr{\Phi}{V}
		=\lim_{j\to\infty}\restr{G_j}{V}
		=\lim_{j\to\infty}\psi_j
		=\varphi.
		\]
		Adding back the normalization constant completes the proof.
	\end{proof}
	In the following, we give a proof of \cref{thm:strict}.
	\begin{proof}[Proof of \cref{thm:strict}]
		Let $\varphi\in \mbox{PSH}(X,\omega)$ such that $\omega_V+dd^c\varphi\geq \varepsilon\omega_V$. For any $0<\varepsilon'<\varepsilon$, we let $\omega'=(1-\varepsilon')\omega$, and then apply \cref{thm:main} to get a function $\Phi\in \mbox{PSH}(X,\omega')$, such that $\Phi|_V=\varphi$, then it is obvious that $\omega+dd^c\Phi=\omega'+\varepsilon'\omega+dd^c\Phi\geq \varepsilon'\omega$ on $X$. This completes the proof of \cref{thm:strict}.
	\end{proof}

\end{document}